\documentclass[12pt, reqno]{amsart}

\usepackage{graphics, stackrel}
\usepackage{amsmath, amssymb, amsthm}
\usepackage{graphicx}
\usepackage{verbatim}
\usepackage{amsfonts}
\usepackage{natbib}
\usepackage{enumitem}

\usepackage{caption}
\usepackage{subcaption}

\usepackage{booktabs}

\usepackage{fancyvrb}
\usepackage[usenames,dvipsnames,svgnames,table]{xcolor}
\usepackage{mdwlist}

\usepackage[citecolor=blue, colorlinks=true, linkcolor=blue]{hyperref}

\usepackage{enumitem}
\setlist[enumerate]{itemsep=2pt,topsep=3pt}
\setlist[itemize]{itemsep=2pt,topsep=3pt}
\setlist[enumerate,1]{label=(\roman*)}

\usepackage{mathrsfs}  
\usepackage{bbm}
\usepackage{bm}        

\usepackage[left=1.25in, right=1.25in, top=1.0in, bottom=1.15in, includehead, includefoot]{geometry}

\renewcommand{\leq}{\leqslant}
\renewcommand{\geq}{\geqslant}

\providecommand{\inner}[1]{\left\langle{#1}\right\rangle}

\usepackage[ruled, linesnumbered]{algorithm2e}

\newcommand\blfootnote[1]{%
  \begingroup
  \renewcommand\thefootnote{}\footnote{#1}%
  \addtocounter{footnote}{-1}%
  \endgroup
}

\newcommand{\la}{\langle}
\newcommand{\ra}{\rangle}

\newcommand{\mM}{\mathcal M}

\newcommand{\pP}{\mathring P}

\newcommand{\RR}{\mathbbm R}

\newcommand{\NN}{\mathbbm N}

\renewcommand{\phi}{\varphi}
\renewcommand{\epsilon}{\varepsilon}

\theoremstyle{plain}
\newtheorem{theorem}{Theorem}[section]
\newtheorem{corollary}[theorem]{Corollary}
\newtheorem{lemma}[theorem]{Lemma}
\newtheorem{proposition}[theorem]{Proposition}

\theoremstyle{definition}

\newtheorem{remark}{Remark}[section]

\begin{document}

\title{}

\begin{center}
    \huge
    Unique Solutions to Power-Transformed Affine Systems
    \blfootnote{The authors thank Shu Hu and Jingni Yang for thoughtful
    comments and suggestions.}
 
    \vspace{1em}

    \large
    John Stachurski,\textsuperscript{a} \\ Ole Wilms\textsuperscript{b} and Junnan Zhang\textsuperscript{c} \par \bigskip

    \small
    \textsuperscript{a} Research School of Economics, Australian National University \\
    \textsuperscript{b} Universit\"at Hamburg and Tilburg University  \\
    \textsuperscript{c} Center for Macroeconomic Research at School of
    Economics, and Wang Yanan Institute for Studies in Economics, Xiamen
    University \\ \bigskip

    \normalsize
    \today
\end{center}

\begin{abstract}
    \vspace{1em}
    Systems of the form $x = (A x^s)^{1/s} + b$ arise in a range of economic,
    financial and control problems, where $A$ is a linear operator acting on a
    space of real-valued functions (or vectors) and $s$ is
    a nonzero real value.  In these applications, attention is focused on
    positive solutions. We provide a simple and complete characterization of
    existence and uniqueness of positive solutions under conditions on
    $A$ and $b$ that imply positivity.  
\end{abstract}

\maketitle

\section{Introduction}

The system of equations 
\begin{equation}\label{eq:ab}
    x = Ax + b 
    \quad \text{where } A \text{ is a linear map}
\end{equation}
occurs naturally in many areas of mathematics and statistics.  For
example, an $x$ solving \eqref{eq:ab} is a steady state for 
the elementary vector difference equation $x_{t+1}
= A x_t + b$.  As a second example, consider the discrete Lyapunov
equation from control theory, which takes the form $G X G^\top - X + Q = 0$, where
all elements are matrices and $X$ is the unknown.  Since $X \mapsto G X
G^\top$ is linear, this is also a version of \eqref{eq:ab}.

In this paper, we study the transformed system 
\begin{equation}\label{eq:abs}
     x = (A x^s)^{1/s} + b, 
     \qquad s \in \RR \text{ and } s \not= 0,
\end{equation}
where powers are taken pointwise (e.g., if $x = (x_i) \in \RR^n$, then $x^s =
(x_i^s) \in \RR^n$). Such systems arise in a range of economic and
financial problems, as well as in discrete time dynamic optimization. In these
cases, the interest is typically on positive solutions, since they relate to
prices or physical quantities.  (Moreover, focusing on positive solutions
avoids having to handle fractional powers of negative numbers.)

In this paper we provide a complete characterization of existence and
uniqueness of positive solutions under conditions on $A$ and $b$ that
imply positivity. In particular, we show that, under a regularity condition that is
always satisfied for the finite-dimensional case, the system \eqref{eq:abs}
has a unique (strictly) positive solution if and only if $r(A)^s < 1$, where
$r(A)$ is the spectral radius of $A$. This finding is a natural extension of
the observation that, when $A$ is irreducible and $b$ is positive,
\eqref{eq:ab} has a unique positive solution if and only if $r(A) < 1$.  

We obtain the results stated above by applying a fixed point theorem for
positive concave (or convex) operators on the positive cone. This fixed point
theorem are a small extension of known results originally obtained by
\cite{du1990fixed}. A range of similar fixed point results based on
monotonicity and order concavity (or convexity) are also relevant \citep[see,
e.g.,][etc.]{amann1972number, amann1976fixed, krasnosel2012approximate,
guo1988nonlinear}, although they are not directly applicable in our
case.\footnote{Uniqueness of a positive solution can be guaranteed since the
    right hand side of~\eqref{eq:abs} is increasing and strongly sublinear
    \citep[Theorem 2.2.2]{guo1988nonlinear}. However, existence usually
    requires additional assumptions (see, e.g.,
\cite{krasnosel2012approximate}).}
Our work is also related to \cite{marinacci2019unique}, who study Tarski-type
fixed points of monotone operators with applications in recursive utilities, as
we do. However, they do not study the particular class of equations that we do
(i.e., those of type \eqref{eq:abs}). \cite{marinacci2010unique} and
\cite{borovivcka2020necessary} provide related results on existence and
uniqueness of recursive utilities.

Studies that directly tackle \eqref{eq:abs} are scarce. However, if we
write~\eqref{eq:abs} as $x = Bx + b$ where $B$ is allowed to be nonlinear,
then we obtain a generalization of operator equations studied by
\cite{zhai2008positive, zhai2010positive, berzig2013positive,
  zhai2017varphi}, among others. In this line of work, $B$ is typically
assumed to be $\alpha$-concave\footnote{An operator $B$ is $\alpha$-concave
  if for any $t \in (0, 1)$, there exists $\alpha \in (0, 1)$ such that
  $B(tx) \geq t^{\alpha}Bx$.}, a property not satisfied in our case when $Bx
= (Ax^s)^{1/s}$. To the best of our knowledge, the most closely related
result in this literature is Theorem 2.6 of \cite{zhai2008positive}, which
provides sufficient conditions for existence and uniqueness of positive
solutions when $B$ is homogeneous. (Despite its generality, they make
assumptions on the operator $B$ that are not applicable to our problem.)

In all of the work mentioned above, only sufficient conditions are provided.
In contrast, our aim is to give exact necessary and sufficient conditions for
existence and uniqueness of positive solutions. 

\section{Fixed Point Theory}\label{s:t}

Let $(B, \| \cdot \|, \leq)$ be a Banach lattice with positive cone $P$ (see,
e.g., \cite{meyer2012banach}). By definition, we have $x \leq y$ if and only
if $y - x \in P$. We assume throughout that $P$ is solid (i.e., has nonempty
interior), and denote the interior of $P$ by $\pP$. For $x, y \in B$, we write $x \ll y$ if $y - x \in \pP$. The
expression $x \ll y$ will sometimes be written $y \gg x$ with identical
meaning. Obviously (a) $y \in \pP$ if and only if $0 \ll y$ and (b) $x \ll y$
implies $x \leq y$.  In what follows we call elements of $\pP$ \emph{strictly
positive}.

Let $E$ be an arbitrary subset of $B$.  A function $G \colon E \to B$
is called \emph{order-preserving} on $E$ if $x, y \in E$ and $x \leq y$ implies $Gx \leq Gy$.
When $E$ is convex, the function $G$ is
called \emph{convex} on $E$ if
    $G(\lambda x + (1-\lambda) y) \leq \lambda Gx + (1-\lambda) Gy$
for all $x, y \in E$ and $\lambda$ in $[0,1]$.
$G$ is called \emph{concave} on $E$ if $-G$ is convex on $E$.
We call $G$ a self-map on $E$ if $x \in E$ implies $Gx \in E$. We
define a self-map $G$ on $E$ to be \emph{globally stable} on $E$ if $G$
has a unique fixed point $\bar x$ in $E$ and, for each $x \in E$,
we have $\| G^k x - \bar x \| \to 0$ as $k \to \infty$.

We let $B'$ be the topological dual space of $B$ and take $P'$ to be the set
of all positive linear functionals on $B$; that is, all $x' \in B'$ such that
$\inner{x', x} \geq 0$ for all $x \in P$.  The next lemma follows from
Corollary~2.8 of \cite{glueck2020almost}:

\begin{lemma}\label{l:glueck}
    In our setting, the following statements are equivalent:
    \begin{enumerate}
        \item $x \gg 0$.
        \item $\inner{x', x} > 0$ for all nonzero $x' \in P'$. 
        \item $\cup_{n \in \NN} [-nx, nx] = B$.
    \end{enumerate}
\end{lemma}

Lemma~\ref{l:glueck} is useful for establishing basic properties of $\pP$.
For example,

\begin{lemma}\label{l:el}
    For elements $x, y, z \in B$, the following statements are true:
    \begin{enumerate}
        \item $x \geq 0$ and $y \gg 0$ implies $x + y \gg 0$.
        \item $x \geq y$ and $y \gg z$ implies $x \gg z$.
        \item $x \gg y$ and $y \geq z$ implies $x \gg z$.
        \item $\pP$ is a sublattice of $B$; that is, $x, y \gg 0$ implies $x \wedge y \gg 0$  and $x \vee y \gg 0$.
    \end{enumerate}
\end{lemma}

\begin{proof}
    Regarding (i), fix $x, y \in B$ with $x \geq 0$ and $y \gg 0$.  Fix
    nonzero $x' \in P'$.   Applying Lemma~\ref{l:glueck}, we have
        $\inner{x', x + y} = \inner{x', x} + \inner{x', y} > 0$.
    Hence $x \gg 0$.

    Regarding (ii), given nonzero $x' \in P'$, we have 
    $x = (x - y) + (y - z)$ and $x \gg 0$ follows from (i).  The proof of
    (iii) is similar.

    For the final claim, fix $x, y \gg 0$.  Since $x \vee y \geq x$,  the
    supremum is strictly positive. For the case of the infimum, let $z = x \wedge y$
    and pick any $b \in B$.  By Lemma~\ref{l:glueck}, for $n$ sufficiently
    large, we have $-nx \leq b \leq nx$ and $-ny \leq b \leq ny$.  
    It follows
    directly that $b \leq (nx) \wedge (ny) = n (x \wedge y) = nz$. Also, $nx,
    ny \geq -b$, so $(nx) \wedge (ny) \geq -b$, or $-nz \leq b$. In
    particular, there exists an $n \in \NN$ with $-nz \leq b \leq nz$.
    Applying Lemma~\ref{l:glueck} yields $z \gg 0$.
\end{proof}

\begin{theorem}\label{t:intconcave}
    Let $G$ be an order-preserving concave self-map on $\pP$. If
    \begin{enumerate}
        \item for all $x \gg 0$, there exists a $p \gg 0$ such that $p \leq
            x$ and $Gp \gg p$, and
        \item for all $x \gg 0$, there exists a $q \gg 0$ such that $x \leq
            q$ and $Gq \leq q$,
    \end{enumerate}
    then $G$ is globally stable on $\pP$.
\end{theorem}

\begin{proof}
    Fix $x \in \pP$, which is nonempty by assumption.  Choose $p, q \in \pP$ as
    in (i)--(ii) above.  Evidently $G$ is an order-preserving concave self-map on $[p,
    q]$. By these facts, the condition $Gp \gg p$ and Theorem~3.1 of
    \cite{du1990fixed}, $G$ has a unique fixed point $\bar x$ in $[p, q]$.
    Since $0 \ll p \leq \bar x$, we have $\bar x \in \pP$ by Lemma~\ref{l:el}.

    Now pick any $y \in \pP$.  We claim that $G^k y \to \bar x$ as $k \to
    \infty$. To see this, we use Lemma~\ref{l:el} to obtain $y \wedge \bar x \gg 0$,
    which means we can find $c \gg 0$ such that $Gc \gg c$ and $c \leq y
    \wedge \bar x$. In addition, $y \vee \bar x \gg 0$, so we can take $d \gg
    0$ such that $Gd \leq d$ and $y \vee \bar
    x \leq d$.  As a result,
    \begin{equation*}
        c \leq y \wedge \bar x \leq y, \bar x \leq y \vee \bar x \leq d.
    \end{equation*}
    In particular, both $y$ and $\bar x$ lie in $[c, d]$.  Moreover, applying
    the same result of \cite{du1990fixed} again, we see that $\bar x$ is the
    only fixed point of $G$ in $[c, d]$ and, moreover, $G^k y \to \bar x$
    as $k \to \infty$.

    We have now shown that $\bar x$ is a fixed point of $G$ in $\pP$
    and $G^k y \to \bar x$ for
    all $y \in \pP$.  This implies that $\bar x$ is the unique fixed point of $G$
    in $\pP$.  In particular, $G$ is globally stable on $\pP$.
\end{proof}

In some instances $G$ is defined over all of the positive cone $P$ and we are
concerned about properties of $G$ on the boundary of the cone.  The next
result provides conditions on $G$ that extend global stability to all of $P
\setminus \{0\}$.

\begin{theorem}\label{t:nonintconcave}
    Let $G$ be an order-preserving concave self-map on $P$ such that conditions
    (i)--(ii) from Theorem~\ref{t:intconcave} hold. If, in addition, for each 
    nonzero $x \in P$, there exists an $m \in \NN$ such that $G^m x \gg 0$,
    then $G$ is globally stable on $P \setminus \{0\}$.
\end{theorem}

\begin{proof}
    First we note that $G$ maps $\pP$ to itself. To see this, fix $x \in \pP$
    and use (i) to choose a $p \gg 0$ with $p \leq x$ and $Gp \gg p$.  By
    isotonicity of $G$ we have $Gx \geq Gp$.  Since $Gp \geq p \gg 0$, it
    follows that $Gx \in \pP$.  In view of Theorem~\ref{t:intconcave}, $G$ is
    globally stable when restricted to $\pP$, with unique fixed point
    $\bar x$.

    Now fix nonzero $x \in P$. Our proof will be complete if we can find an
    $\alpha \in (0,1)$ and an $M < \infty$ such that $\| G^k x - \bar x\| \leq
    \alpha^k M$ for all $k \in \NN$.

    By hypothesis, we can choose an $m \in \NN$ such that $x_m := G^m x \gg 0$.
    Also, since $G$ is globally stable when restricted to $\pP$, we can
    take a $\beta \in (0,1)$ and $N < \infty $ such that $\| G^k x_m - \bar
    x\| \leq \beta^k N$ for all $k \in \NN$.  Let $e_i = \| G^i x - x_m \|$
    for $i=1, \ldots, m$ and let $e = \max_{i \leq m} e_i$.
    If we set $\alpha := \beta$ and $M := \beta^{-m} \max\{N, e\}$, it is easy
    to verify that $\| G^k x - \bar x\| \leq
    \alpha^k M$ for all $k \in \NN$.
    We conclude that $G$ is globally stable on $P \setminus \{0\}$, as
    claimed.
\end{proof}

Next we turn to the convex case.

\begin{theorem}\label{t:convex}
    Let $G$ be an order-preserving convex self-map on the positive cone $P$.  Suppose that
    \begin{enumerate}
        \item $G0 \gg 0$ and
        \item for all $x \in \pP$, there exists a $b \in P$ such that $x \leq
            b$ and $Gb \ll b$.
    \end{enumerate}
    If these two conditions hold, then $G$ is globally stable on $\pP$.
\end{theorem}

\begin{proof}
    Fix $x \in \pP$ and choose $b \in P$ as in (ii) above.  Since $G$ is an
    order-preserving convex self-map on $[0, b]$ and $Gb \ll b$, Theorem~3.1 of
    \cite{du1990fixed} implies that $G$ has a unique fixed point $\bar x$ in $[0,
    b]$.  Moreover, $0 \leq \bar x$, so isotonicity of $G$ implies $0 \ll G0 \leq
    G\bar x = \bar x$. Hence $\bar x \in \pP$.

    Now pick any $y \in \pP$.  We claim that $G^k y \to \bar x$.
    Since $y \vee \bar x \in \pP$, so we can take $d \in P$ such that $Gd \ll d$
    and $y \vee \bar x \leq d$.  Note that $0 \leq y, \bar x \leq y \vee \bar
    x\leq
    d$. In particular, $y$ and $\bar x$ are in $[0, d]$.  Applying the same
    result of \cite{du1990fixed} again, this time to $G$ on $[0, d]$, we see
    that $\bar x$ is the only fixed point of $G$ in $[0, d]$ and, moreover, $G^k
    y \to \bar x$.  This last result also implies that $\bar x$ is the
    only fixed point of $G$ in $\pP$.
    Evidently $G$ maps $\pP$ to itself, so the proof is now done.
\end{proof}

\section{Power-Transformed Affine Systems}

In this section we apply the fixed point results of Section~\ref{s:t} to the
class of nonlinear equations discussed in the introduction.  Below, in
Section~\ref{s:ea}, we will discuss how these problems arise in applications.
Here we focus on characterizing existence and uniqueness of positive
solutions.

\subsection{Environment}

Returning to the system \eqref{eq:abs}, we assume throughout that the unknown
object $x$ takes values in the set $C(T)$ of continuous real-valued
functions on compact Hausdorff space $T$, and that $A$ is a linear
operator from $C(T)$ to itself. The set $C(T)$ is a Banach lattice when
paired with the supremum norm and the usual pointwise partial order. In line
with Section~\ref{s:t}, we let $P$ be the positive cone of $C(T)$ and $\pP$
be the interior of $P$. An \emph{ideal} in $C(T)$ is a vector subspace $I
\subset C(T)$ such that $g \in I$ and $f \in C(T)$ with $|f| \leq |g|$
implies $f \in I$.

As in the introduction, powers are pointwise operations, so that $x^s$ is the
function $T \ni t \mapsto x^s(t)$. For the powers in \eqref{eq:abs} to make
sense we operate only on positive elements of $C(T)$, as clarified below.
Obviously, our treatment includes the finite-dimensional case.  (If $T$ is
finite with $n$ elements, then we endow $T$ with the discrete topology, in
which case $C(T)$ is isometrically isomorphic to $\RR^n$ and \eqref{eq:abs}
can be viewed as a system of $n$ equations in $\RR^n$.)

A linear operator $A \colon C(T) \to C(T)$ is called \emph{positive} if $A$ is
a self-map on $P$ and \emph{eventually compact} if there exists a $k \in \NN$
such that $A^k$ is compact (i.e., maps the unit ball of $C(T)$ to a relatively
compact subset of $C(T)$). A positive linear operator $A$ is called
\emph{irreducible} if the only nontrivial closed ideal on which $A$ is
invariant is the whole space $C(T)$. The spectral radius $r(A)$ of $A$ can be
defined by Gelfand's formula $r(A) = \lim_{k \to \infty} \| A^k \|^{1/k}$,
where $\| \cdot \|$ is the operator norm on the bounded linear operators from
$C(T)$ to itself.

\subsection{Equivalency}

Let $A \colon C(T) \to C(T)$ be a positive linear operator and suppose $b \in
C(T)$ and $s \in \RR$ with $s \not=0$ and $b \gg 0$. In studying
\eqref{eq:abs}, we also handle the representation
\begin{equation}\label{eq:yrep}
    y = ((Ay)^{1/s} + b)^s,
\end{equation}
which is obtained from \eqref{eq:abs} via the change of variable $y = x^s$.
The system \eqref{eq:yrep} is often easier to handle than the
original system \eqref{eq:abs}, largely because the transformation on the
right is decomposed into one purely linear operation (i.e., $y \mapsto Ay$) and
one nonlinear operation.

We solve the power-transformed affine systems \eqref{eq:abs} and \eqref{eq:yrep}
by converting them into fixed point problems associated with the self-maps $F,
G$ on $\pP$ defined by 
\begin{equation*}
    Fx = (A x^s)^{1/s} + b 
    \quad \text{and} \quad
    Gy = ((Ay)^{1/s} + b)^s
\end{equation*}
Let $H$ be the homeomorphism from $\pP$ to itself defined by $Hx = x^s$.
Then $F$ and $G$ are topologically conjugate under $H$, in the sense that $H\circ F
= G\circ H$.  As a result, $F$ has a unique fixed point in $\pP$ if and only if the
same is true for $G$, and, moreover, fixed points of $F$ and $G$ have the same
stability properties.

\subsection{Positive Solutions}

In what follows,
if $x \in P$ and $x$ satisfies \eqref{eq:abs} then we call $x$ a
\emph{positive} solution to \eqref{eq:abs}.  If, in addition, $x \gg 0$, then
we call $x$ a \emph{strictly positive} solution.  Analogous terminology is
used for solutions to \eqref{eq:yrep}.

\begin{theorem}\label{t:tlbk}
    Suppose $b \gg 0$ and $s \in \RR$ with $s \not=0$.  If
    $A$ is irreducible and eventually compact, then the following statements
    are equivalent:
    \begin{enumerate}
        \item $r(A)^s < 1$.
        \item $G$ is globally stable on $\pP$.
    \end{enumerate}
    Moreover, if $r(A)^s \geq 1$, then $G$ has no fixed point in $\pP$
\end{theorem}

As an immediate consequence, we have

\begin{corollary}\label{c:bk}
    Under the conditions on $A, b$ and $s$ stated in Theorem~\ref{t:tlbk},
    both \eqref{eq:abs} and \eqref{eq:yrep} have a unique strictly positive
    solution if and only if $r(A)^s < 1$.
\end{corollary}

\begin{remark}\label{r:ir}
    Since $A$ is irreducible we have $r(A) > 0$ (see, e.g.,
    \cite{meyer2012banach}, Lemma~4.2.9).  Hence we can rewrite $r(A)^s < 1$
    as $s \cdot \ln r(A) < 0$.  Thus, condition (ii) is equivalent
    to the statement that $r(A) \not= 1$ and $s$ and $\ln r(A)$ have opposite signs.
\end{remark}

\subsection{Proof of Theorem~\ref{t:tlbk}}

We assume throughout that the conditions on $A, s$ and $b$ imposed in
Theorem~\ref{t:tlbk} hold. 

\begin{proposition}\label{p:sp}
    $G$ is order-preserving on $\pP$ and has the following shape properties:
    \begin{enumerate}
        \item $G$ is concave on $\pP$ whenever $s \in \RR \setminus [0, 1)$.
        \item $G$ is convex on $\pP$ whenever $s \in (0, 1]$.
    \end{enumerate}
\end{proposition}

\begin{proof}
    Fix $t \in T$ and define
    \begin{equation*}
        \phi_t(r) := \left[ r^{1/s} + b(t) \right]^s \qquad (r > 0).
    \end{equation*}
    Then, for any $y \in \pP$, $(Gy)(t)$ can be written as $\phi_t[(Ay)(t)]$.

    We first show that $G$ is order-preserving. Since $A$ is linear, irreducible (and
    hence positive), $A$ is order-preserving on $C(T)$. It is easy to see
    that $\phi_t$ is an increasing function for all $t \in T$, so 
    $t \in T$ and $y, z \in \pP$ with $y \leq z$ implies
    \begin{equation*}
        (Gy)(t) = \phi_t[(Ay)(t)] \leq \phi_t[(Az)(t)] = (Gz)(t).
    \end{equation*}
    In particular, $G$ is order preserving on $\pP$.

    To see that $G$ is concave on $\pP$ when $s < 0$ or $s \geq 1$, we fix $y,
    z \in \pP$ and $\lambda \in [0, 1]$ and let $h := \lambda y + (1 -
    \lambda)z$. Observe that $\phi_t$ is concave in this case for all $t \in
    T$. Hence, fixing $t \in T$ and using linearity of $A$,
    \begin{align*}
        \phi_t [(Ah)(t)] 
        & = \phi_t[\lambda (Ay)(t) + (1 - \lambda) (Az)(t)] \\
        & \geq \lambda \phi_t[(Ay)(t)] + (1 - \lambda)\phi_t[(Az)(t)].
    \end{align*}
    Hence $(Gh)(t) \geq \lambda(Gy)(t) + (1 - \lambda) (Gz)(t)$. Since $t \in
    T$ was arbitrary, we conclude that $G$ is concave on $\pP$. 

    Similarly, (ii) follows from the fact that each $\phi_t$ is convex when $s
    \in (0, 1]$.
\end{proof}

\begin{proposition}\label{p:bkn}
    If $G$ has a fixed point in $\pP$, then $r(A)^s < 1$.
\end{proposition}

\begin{proof}
    Let $\mM$ be the topological dual space for $C(T)$. Let $A^*$ be the
    adjoint operator associated with $A$.\footnote{The adjoint of a bounded
      linear operator $A: C(T) \to C(T)$ is given by $A^*: \mM \to \mM$ such
      that $\la A^* \mu, x\ra = \la \mu, Ax \ra$ for all $x
      \in C(T)$ and all $\mu \in \mM$.} Since $A$ is irreducible and
    eventually compact, Lemma~4.2.11 of \cite{meyer2012banach} ensures us
    existence of an $e^* \in \mM$ such that
    \begin{equation}\label{eq:pao}
        \inner{e^*, x} > 0 \text{ for all } x \gg 0
        \quad \text{and} \quad
        A^* e^* = r(A) e^*.
    \end{equation}
    Let $\bar y$ be a fixed point of $G$ in $\pP$.

    First consider the case where $s > 0$. In this case we have $\phi_t [y(t)] >
    y(t)$ for all $y \in \pP$ and $t \in T$, so
    \begin{equation*}
        \bar y(t) 
        = (G \bar y)(t) 
        = \phi_t [(A \bar y)(t)] 
        > (A \bar y)(t).
    \end{equation*}
    Hence $\bar y \gg A \bar y$.
    Taking $e^*$ as in \eqref{eq:pao}, we then have $\la e^*, A \bar y - \bar y
    \ra < 0$, or, equivalently, $ \la e^*, A \bar y \ra < \la e^*, \bar y
    \ra$. Using the definition of the adjoint and \eqref{eq:pao} gives $r (A)
    \la e^*, \bar y \ra = \la A^* e^*, \bar y \ra = \la e^*, A \bar y \ra$,
    so it must be that $ r (A) \la e^*, \bar y \ra < \la e^*, \bar y \ra$.
    As a result, $r (A) < 1$. Because $s > 0$, we have $r(A)^s < 1$.

    Now consider the case where $s < 0$.  We have $\phi_t [ y(t)] < y(t)$ for all
    $y \in \pP$ and $t \in T$. As a result, 
    \begin{equation*}
        \bar y(t) = (G\bar y)(t) = \phi_t[(A \bar y)(t)] < (A \bar y)(t).
    \end{equation*}
    But then, taking $e^*$ as in \eqref{eq:pao}, we have $\la
    e^*, A \bar y - \bar y \ra > 0$, or, equivalently, $ \la e^*, A \bar y
    \ra > \la e^*, \bar y \ra$. Using \eqref{eq:pao} again gives $r (A) \la
    e^*, \bar y \ra = \la e^*, A \bar y \ra$, so $r(A) \la e^*, \bar y \ra >
    \la e^*, \bar y \ra$. Hence $r(A) > 1$. Because $s < 0$, this yields $r(A)^s < 1$.
\end{proof}

\begin{proposition}\label{p:inada}
    If $r(A)^s < 1$, then, for each $y \in \pP$, there exists a pair
    $p, q \in \pP$ such that $p \leq y \leq q$, $Gp \gg p$ and $Gq \ll q$. 
\end{proposition}

In the proof of Proposition~\ref{p:inada}, we set $r := r(A)$. We use the fact
that $r > 0$, and that there eixsts a dominant eigenvector $e \in \pP$ such that $Ae = r(A)
e$, as follows from irreducibility and eventual compactness of $A$ (see, e.g., 
Lemma~4.2.14 of \cite{meyer2012banach}.)

\begin{proof}
    Observe that, for each $c \in (0, \infty)$, we have 
    \begin{equation*}
        \frac{G(ce)}{ce} 
        = \left(
            \frac{(c re)^{1/s} + b}{(ce)^{1/s}}
          \right)^s
        = \left(
            r^{1/s} + \frac{b}{(ce)^{1/s}}
          \right)^s.
    \end{equation*}
    Now consider the case where $s < 0$ and $r > 1$. Recall that $T$ is
    compact and $b \gg 0$. Then 
    $c \to 0$ implies that $(G(ce))/(ce) \to r$ uniformly, which in turn
    implies 
    \begin{equation}\label{eq:czero}
        \exists \, c_0 > 0 \text{ and } \delta_0 > 1
        \text{ such that } c \leq c_0 \implies G(c e) \geq \delta_0 c e.
    \end{equation}
    Also, if $c \to \infty$, then $(G(ce))/(ce) \to 0$ uniformly, which in turn
    implies 
    \begin{equation}\label{eq:cone}
        \exists \, c_1 > 0 \text{ and } \delta_1 < 1
        \text{ such that } c \geq c_1 \implies G(c e) \leq \delta_1 c e.
    \end{equation}
    Now suppose $s > 0$ and $r < 1$. In this case, $c \to 0$ implies that
    $(G(ce))/(ce) \to \infty$ uniformly, and hence \eqref{eq:czero} holds.
    Also, in the same setting, $c \to \infty$ implies $(G(ce))/(ce) \to r$
    uniformly, and since $r < 1$ we have \eqref{eq:cone}.

    So far we have shown that \eqref{eq:czero} and \eqref{eq:cone} are both
    valid when $r^s < 1$.  Now fix $y \in \pP$.  Since $e, y \gg 0$ and $T$ is
    compact, we have $c e \ll y$ for all sufficiently small $c \in (0,
    \infty)$.  Similarly, $c e \gg y$ for all sufficiently large $c$.
    Combining these facts with \eqref{eq:czero} and \eqref{eq:cone}, we obtain
    $p, q \in \pP$ with $p \leq y \leq q$, $Gp \gg p$ and $Gq \ll q$.
\end{proof}

\begin{proof}[Proof of Theorem~\ref{t:tlbk}]
    Let the conditions of Theorem~\ref{t:tlbk} hold, in the sense that $b \gg
    0$, $s \in \RR \setminus \{0\}$ and $A$ is irreducible and eventually
    compact.  If $s \in (0, 1]$ and $r(A) < 1$, then $G$ is order-preserving and convex
    on $\pP$ by Proposition~\ref{p:sp}. Moreover, $G0 = b^s \gg 0$ and, by
    Proposition~\ref{p:inada}, for each $y \in \pP$ there is a $q \in \pP$
    such that $y \leq q$ and $q \ll Gq$.  Hence, by Theorem~\ref{t:convex},
    $G$ is globally stable on $\pP$.

    If, on the other hand, $s < 0$ and $r(A) > 1$ or $s \geq 1$ and $r(A) < 1$, then $G$ is
    order-preserving and concave on $\pP$ by Proposition~\ref{p:sp}. Moreover,
    by Proposition~\ref{p:inada}, for each $y \in \pP$ there is a pair $p, q
    \in \pP$ such that $p \leq y \leq q$, $Gp \gg p$ and $q \ll Gq$.  Hence,
    by Theorem~\ref{t:intconcave}, $G$ is globally stable on $\pP$.

    We have now shown that $r(A)^s < 1$ implies $G$ is globally stable on
    $\pP$.  For the converse we apply Proposition~\ref{p:bkn}, which tells us
    that $G$ has no fixed point in $\pP$ when $r(A)^s \geq 1$.
\end{proof}

\section{Applications}\label{s:ea}

In this section we give several applications of the preceding results on
power-transformed affine systems.

\subsection{State-Dependent Discounting}

\cite{toda2019wealth} studies an optimal consumption problem and shows that
optimal consumption from current wealth $w$ takes the form
\begin{equation*}
     c(z) = b(z)^{-1/\gamma} w
\end{equation*}
where $\gamma > 0$ is a parameter and $b$ is a function of an uncertainty
state $z \in Z$ satisfying the equation
\begin{equation}\label{eq:toda}
    b(z) = \left\{
        1 + \left[\beta(z) R(z)^{1-\gamma} (Q b)(z)\right]^{1/\gamma} 
    \right\}^\gamma.
\end{equation}
Here $\beta(z) > 0$ is the discount factor in state $z$, $R(z) > 0$ is the gross
interest rate in state $z$ and $Qb$ is defined at each $z \in Z$ by 
\begin{equation*}
    (Qb)(z) = \sum_{z' \in Z} b(z') q(z, z'),
\end{equation*}
where $q$ is the transition matrix of an irreducible Markov chain.
\cite{toda2019wealth} shows that \eqref{eq:toda} has a strictly positive
solution if and only if the spectral radius of the linear operator $A$ defined
by 
\begin{equation}\label{eq:defk}
    (Af)(z) := \beta(z) R(z)^{1-\gamma} \sum_{z' \in Z} f(z') q(z, z')
\end{equation}
has a spectral radius less than unity.  In the setting of
\cite{toda2019wealth}, the set $Z$ is finite.

We can prove the same result using Corollary~\ref{c:bk}, while also adding
uniqueness.  Indeed, with $A$ defined as in \eqref{eq:defk}, we can express
\eqref{eq:toda} as
\begin{equation}\label{eq:toda2}
    b(z) = \left\{ 1 + \left[(A b)(z)\right]^{1/\gamma} \right\}^\gamma.
\end{equation}
Since $q$ is irreducible and $\beta$ and $R$ are strictly positive, the
operator $A$ is irreducible.  Eventual compactness of $A$ is immediate from
finiteness of $Z$. Hence Corollary~\ref{c:bk} implies that \eqref{eq:toda2}
has a unique strictly positive solution if and only if $r(A)^\gamma < 1$.
Since $\gamma > 0$, this reduces to $r(A) < 1$, which is the same condition
used by \cite{toda2019wealth}.

Notice that Corollary~\ref{c:bk} can also be used to relax the finiteness
assumption on $Z$ by assuming instead that $Z$ is a compact Hausdorff space.
In this case continuity restrictions need to be placed on $\beta$ and $R$.

\subsection{Recursive Preferences}

In finance, Epstein--Zin preferences have become increasingly important for
specifying and solving mainstream asset pricing models (see, e.g.,
\cite{hansen2012recursive}, \cite{bansal2012empirical},
\cite{schorfheide2018identifying} or \cite{gomez2021important}). Similar
specifications have also been adopted in macroeconomics (see, e.g.,
\cite{basu2017uncertainty}).  A generic specification takes the form
\begin{equation}\label{eq:ez}
    v = ((1-\beta) c^\rho + \beta (Rv)^\rho)^{1/\rho}
\end{equation}
where $c \in \pP \subset C(T)$ denotes consumption in each state and $\rho$
and $\beta$ are positive parameters.  The operator $R$ is defined at $v \in
\pP$ by $Rv = (Qv^\alpha)^{1/\alpha}$, where $\alpha$ is a nonzero parameter
and $Q$ is an irreducible Markov operator. The unknown function $v$ gives
lifetime utility at each state of world.

It is known (see, e.g., \cite{borovivcka2020necessary}) that \eqref{eq:ez} has
a unique positive solution if and only if $\beta < 1$.  This result can be
derived from our results by setting $w = v^\alpha$, $s = \alpha/\rho$, and rewriting
\eqref{eq:ez} as
\begin{equation*}
    w = 
    \left\{
        (1-\beta) c^\rho
        + \beta (Qw)^{\rho/\alpha}
    \right\}^{\alpha/\rho}
    =
    \left\{
        (1-\beta) c^\rho
        + (\beta^s Qw)^{1/s}
    \right\}^s.
\end{equation*}
By Corollary~\ref{c:bk}, a unique strictly positive solution exists if and
only if the linear operator $A := \beta^s Q$ is such that $r(A)^s < 1$, which
can also be written as $r(A)^{1/s} < 1$.  Since $Q$ is a Markov operator we
have $r(Q) = 1$, and hence $r(A)^{1/s} < 1$ if and only if $\beta < 1$.

\subsection{Wealth-Consumption Ratio}

Another important object in asset pricing is the wealth-consumption ratio of
the representative agent. In a similar environment to the one described
above, the equilibrium wealth-consumption ratio $w$ under Epstein-Zin
preferences satisfies the following first-order condition
\begin{equation}
    \label{eq:wc}
    \beta^s Qw^s = (w - 1)^s,
\end{equation}
where the operator $Q$ is assumed to be an irreducible and eventually compact
linear operator that contains information about consumption growth. Letting
$A := \beta^s Q$, we can rewrite \eqref{eq:wc} as
\begin{equation*}
    w = (A w^s)^{1/s} + 1,
\end{equation*}
which has the same form as~\eqref{eq:abs}. By Corollary~\ref{c:bk}, a unique
strictly positive wealth-consumption ratio exists if and only if the linear
operator $A$ satisfies $r(A)^s < 1$, or equivalently, $\beta r(Q)^{1/s} < 1$.

\subsection{Growth with CES Production Function}

Consider a discrete time multi-sector growth model with total depreciation
and no population growth or technological progress. The law of motion for the
capital-labor ratio is given by $k_{t+1} = s f(Ak_t)$, where $s$ is a savings
rate, $f$ is a CES production function, $A$ is an irreducible matrix that
characterizes the technology, and $k_t$ is a vector of multi-sector
capital-labor ratios. Then the steady state capital-labor ratio satisfies
\begin{equation}
    \label{eq:ces}
    k = s \left\{\theta + (1-\theta) (A k)^{\rho}\right\}^{1/\rho},
\end{equation}
where $\theta \in (0, 1)$, $\rho \neq 0$, and all algebraic operations are
performed elementwise. We can rewrite~\eqref{eq:ces} as
\begin{equation*}
    k = \left\{ s^\rho \theta + \left(s(1 - \theta)^{1/\rho} Ak\right)^\rho
    \right\}^{1/\rho}.
\end{equation*}
Then by Corollary~\ref{c:bk}, a unique and strictly positive vector of
multi-sector capital-labor ratios exists in the steady state if and only if
$s^\rho (1 - \theta) r(A) < 1$.

\bibliographystyle{ecta}
\bibliography{localbib}

\end{document}